\documentclass[a4paper, 12pt,oneside,reqno]{amsart}

\usepackage[a4paper]{geometry}
\usepackage{cmap}
\usepackage[T2A]{fontenc}
\usepackage[english]{babel}
\usepackage[utf8]{inputenc}
\usepackage{xcolor}

\usepackage{amsmath, amssymb,  amsthm, amscd, hyperref}

\def\Lalg{L}

\def\Var{\mathrm{Var}}
\def\Lie{\mathrm{Lie}}
\def\Com{\mathrm{Com}}
\def\Pois{\mathrm{Pois}}
\def\Perm{\mathrm{Perm}}
\def\LSym{\mathrm{LSym}}
\def\As{\mathrm{As}}

\def\End{\mathop{\fam 0 End}\nolimits}
\def\Hom{\mathop{\fam 0 Hom}\nolimits}

\def\ChEnd{\mathop{\fam 0 ChEnd}\nolimits}
\def\oo#1{\mathbin{ {}_{#1}}}

\title{Chiral algebras with abelian conformal part}

\newtheorem{proposition}{Proposition}
\newtheorem{theorem}{Theorem}

\newtheorem{remark}{Remark}

\theoremstyle{definition}

\newtheorem{definition}{Definition}
\newtheorem{example}{Example}

\author{I. V. DUDIN and P. S. KOLESNIKOV$^*$}

\thanks{$^*$Partially supported by Russian Science Foundation, project 25-41-00005}

\begin{document}

\begin{abstract}

We study a categorical approach to the concept of varieties of chiral algebras.
We prove that the class of chiral algebras in the variety defined by a binary quadratic operad $\Var$, whose conformal structure is abelian, coincides with the class of differential algebras in the variety defined by the Manin black product of the operads $\Var$ and $\Com$, where
$\Com $ is the operad of associative commutative algebras.

\end{abstract}

\maketitle

\section*{Introduction}

The notion of a vertex algebra (vertex operator algebra) emerged as a tool to describe the properties of coefficients in the operator product expansion (OPE) in two-dimensional conformal field theory \cite{BPZ}. In addition, vertex algebras have found applications in representation theory \cite{Borch} (see also \cite{FLM-MonsterBook}). The classical definition of a vertex algebra in terms of formal series formed by linear operators on a vector space equipped with a translation operator does not clearly reflect the underlying algebraic nature of this concept. The categorical approach to vertex and conformal algebras proposed in \cite{BD-Chiral}, although technically more involved, makes it possible to treat ordinary Lie algebras over a field, conformal Lie algebras, and vertex algebras within a unified framework, namely as morphisms from the operad $\Lie$ to the corresponding endomorphism operad (ordinary linear, conformal, or chiral).

The theory of  Lie conformal algebras \cite{Kac_VA_Beginn} has provided a source of problems and methods for numerous studies of conformal algebras in other varieties (associative, Jordan, left-symmetric, Novikov, etc.). The categorical definition of a conformal algebra as a special case of a pseudoalgebra \cite{BDK-2001} serves in these investigations as a key tool for the proper formulation of problems and the choice of methods for their solution.

In \cite{BDKH-2019}, various equivalent approaches to the definition of a vertex algebra are studied in detail. In particular, the notion of the operad of chiral endomorphisms $\ChEnd_V$ of a vector space $V$ equipped with a translation operator $\partial$ is presented there in a purely algebraic form. In Section~1, we provide a more detailed exposition of this construction, which is necessary for subsequent computations. The structure of a  Lie chiral algebra on the space $V$ is defined by a morphism of operads
\[
\Lie \to \ChEnd_V.
\]
As shown in \cite{BDKH-2019}, the difference between a  Lie chiral algebra and a vertex algebra in the usual sense (see, for example, \cite{F-BZvi}) is only that a  Lie chiral algebra contains no information about the vacuum vector.

If the operad $\Lie$ is replaced by the operad $\As$ corresponding to the variety of associative algebras, the resulting class of objects (associative chiral algebras) turns out to be ``degenerate'' in a certain sense: any algebraic expression in an associative chiral algebra containing more than one chiral multiplication (an operation analogous to the normally ordered product in vertex algebras) can be expressed in terms of the conformal part of the algebra.

The aim of this work is, in particular, to establish a criterion that allows one to immediately determine which binary quadratic operads $\Var$ yield a degenerate class of $\Var$-chiral algebras.
To that end, we consider the variety of ordinary algebras that are $\Var$-chiral algebras with an abelian conformal part.

It turns out that for any binary quadratic operad $\Var$, the class of all $\Var$-chiral algebras with an abelian conformal part coincides with the class of differential algebras in the variety defined by the operad $\Var \bullet \Com$, where $\bullet$ denotes the Manin black product of operads \cite{GK-94}.

\section{The operad of chiral endomorphisms}

In this section, we present the construction of the operad of chiral endomorphisms (chiral operad) of a given module $V$ equipped with a linear operator $\partial$, and provide the necessary examples following \cite{BDKH-2019}. To do this, it is necessary to define a certain family of representations $\mathcal O_n$ of the algebras of differential operators $\mathcal D_n$ for $n \ge 1$.

Let $\Bbbk$ be a field of characteristic zero, and let $z_1,\dots, z_n$ be formal variables. For brevity, we introduce the following notation: $z_{ij}=z_i-z_j$; here and henceforth, for indices $i,j \in \mathbb{N}$, we assume $1 \leq j < i \leq n$, where $n \in \mathbb{N}$. Consider the algebra of Laurent polynomials in $z_{ij}$
\[
  \mathcal{O}_n = \Bbbk [ z_{ij}, z_{ij}^{-1}]
\]
as a subalgebra of the field of rational functions $\Bbbk (z_1,\dots, z_n)$.

Let $\mathcal D_n$ denote the following algebra of differential operators:
\[
\mathcal D_n = \Bbbk [ z_{ij}] [\partial_{z_1}, \dots , \partial_{z_n}].
\]
In particular, $\mathcal O_1 = \Bbbk $ and $\mathcal D_1 = \Bbbk [\partial_{z_1}]$.

By construction, $\mathcal{D}_n$ is a subalgebra of the $n$th Weyl algebra and, therefore, acts on $\Bbbk (z_1,\dots, z_n)$ as a left module. It is easy to see that $\mathcal{O}_n$ is a $\mathcal{D}_n$-submodule of $\Bbbk (z_1,\dots, z_n)$. It is known (see~\cite[Lemma~6.4]{BDKH-2019}) that $\mathcal{O}_n$, as a $\mathcal{D}_n$-module, is generated by a single element
\[
\omega_{n} = \prod\limits_{1 \leq j < i \leq n} z_{ij}^{-1}.
\]
In particular, $\omega_1 = 1$.

Let $V$ be a vector space equipped with a linear operator
$\partial : V \to V$. On the space $V^{\otimes n} \otimes \mathcal{O}_n$, we define a right $\mathcal{D}_n$-module structure via the following actions:
\begin{gather}
  (v_1 \otimes \cdots \otimes v_n \otimes f)z_{ij} = v_1 \otimes \cdots \otimes v_n \otimes z_{ij}f, \label{eq:Dn-action-left1} \\
  (v_1 \otimes \cdots \otimes v_n \otimes f)\partial_{z_k} = v_1 \otimes \cdots \otimes (\partial v_k) \otimes \cdots \otimes v_n \otimes f - v_1 \otimes \cdots \otimes v_n \otimes \frac{\partial f}{\partial z_k}. \label{eq:Dn-action-left2}
\end{gather}

Let $H$ denote the polynomial algebra $\Bbbk[\partial]$ in the formal variable $\partial$. The space $V$ considered above can be viewed as a left $H$-module, where the action of the element $\partial \in H$ is given by the corresponding operator.

Consider the polynomial algebra $\Lalg_n = \Bbbk[\lambda_1, \dots, \lambda_n]$, $n \ge 1$, in another set of formal variables. Define a right action of $H$ on the algebra $\Lalg_n$ by the rule
\[
    h(\lambda_1 , \dots , \lambda_n) \partial = -(\lambda_1 + \cdots + \lambda_n) h(\lambda_1 , \dots , \lambda_n) ,\quad h \in \Lalg_n.
\]

We also define a right action of $\mathcal{D}_n$ on $\Lalg_n$ as follows:
\[
    h z_{ij} = \left( \frac{\partial}{\partial \lambda_j} - \frac{\partial}{\partial \lambda_i} \right) h,
    \quad
    h \partial_{z_k} = -\lambda_k h,
\]
for $h \in \Lalg_n$, $1 \le j < i \le n$, $1 \le k \le n$.
Note that this action is $\partial$-invariant; therefore, the right $\mathcal{D}_n$-module structure on the space $\Lalg_n \otimes_H V$ is well-defined by the action on the first tensor factor.

Now, using the constructions introduced above, we define a family of vector spaces of {\em chiral endomorphisms} of the $H$-module $V$ as
\[
    \ChEnd_V(n) = \Hom_{\mathcal{D}_n}(V^{\otimes n} \otimes \mathcal{O}_n, \Lalg_n \otimes_H V), \quad n \geq 1.
\]

To write an element $X \in \ChEnd_V(n)$, we will explicitly indicate the sets of variables $z_1,\dots, z_n$, on which the module $\mathcal{O}_n$ is constructed, and $\lambda_1,\dots, \lambda_n$, on which the module $\Lalg_n$ is constructed:
\[
X: v_1 \otimes \dots \otimes v_n \otimes f \mapsto X_{\lambda_1,\dots,\lambda_n}^{z_1,\dots, z_n}(v_1,\dots, v_n; f),
\]
where $v_i \in V$ and $f \in \mathcal{O}_n$.

Since $\mathcal{O}_n$ is a cyclic $\mathcal{D}_n$-module with the generator $\omega_n$, a function $X \in \ChEnd_V(n)$ is completely determined by its values on the arguments $(v_1,\dots,v_n; \omega_n)$, where $v_i \in V$.
Indeed, for $n=2$, due to \eqref{eq:Dn-action-left1}, we have
\[
X_{\lambda_1,\lambda_2}^{z_1,z_2}(u,v; z_{ij}f)
= X_{\lambda_1,\lambda_2}^{z_1,z_2}(u,v; f)z_{ij}
= \left(\dfrac{\partial}{\partial \lambda_j}
 - \dfrac{\partial}{\partial \lambda_i}\right)
X_{\lambda_1,\lambda_2}^{z_1,z_2}(u,v; f)
\]
for $u,v \in V$ and $f \in \mathcal{O}_2$.
Furthermore,
\[
(u \otimes v \otimes \partial_{z_2}f)
= (u \otimes \partial v \otimes f) - (u \otimes v \otimes f)\partial_{z_2}
\]
due to \eqref{eq:Dn-action-left2}; therefore,
\[
X_{\lambda_1,\lambda_2}^{z_1,z_2}(u,v; \partial_{z_2}f)
= X_{\lambda_1,\lambda_2}^{z_1,z_2}(u,\partial v; f) + \lambda_2 X_{\lambda_1,\lambda_2}^{z_1,z_2}(u, v; f)
= X_{\lambda_1,\lambda_2}^{z_1,z_2}(u,(\partial+\lambda_2) v; f).
\]
Similarly,
\[
X_{\lambda_1,\lambda_2}^{z_1,z_2}(u,v; \partial_{z_1}f)
= X_{\lambda_1,\lambda_2}^{z_1,z_2}((\partial+\lambda_1)u, v; f).
\]
Thus, the value of a chiral endomorphism $X \in \ChEnd_V(2)$ on an arbitrary $f \in \mathcal{O}_2$ is expressed in terms of its value on the generator $\omega_2 \in \mathcal{O}_2$.

In the general case, the condition of $\mathcal{D}_n$-linearity of a chiral endomorphism $X$ is expressed by the relations of {\em sesquilinearity}:
\begin{equation} \label{eq:ChEnd-3/2-lin}
\begin{gathered}
X_{\lambda_1, \dots , \lambda_n}^{z_1, \dots , z_n} (v_1, \dots , v_n; \partial_{z_i} f)
=
X_{\lambda_1, \dots , \lambda_n}^{z_1, \dots , z_n} (v_1, \dots , (\partial + \lambda_i)v_i, \dots , v_n; f) ,
\\
X_{\lambda_1, \dots , \lambda_n}^{z_1, \dots , z_n} (v_1, \dots , v_n; z_{ij} f)=
\left(\frac{\partial}{\partial \lambda_j}-\frac{\partial}{\partial \lambda_i}\right )
X_{\lambda_1, \dots , \lambda_n}^{z_1, \dots , z_n} (v_1, \dots , v_n; f).
\end{gathered}
\end{equation}
Therefore, for any $n \ge 1$, it is sufficient to evaluate  a chiral endomorphism on the last argument $\omega_n \in \mathcal{O}_n$.

For $n=1$, the space $\Lalg_1 \otimes_H V$ is isomorphic to $V$, and $\omega_1 = 1$; therefore,
\[
X^{z_1}_{\lambda_1}(u; 1) = X(u) \in V
\]
for any $u \in V$. In this case, the $\mathcal{D}_1$-linearity property \eqref{eq:ChEnd-3/2-lin} implies that
\[
0 = X^{z_1}_{\lambda_1}(u; \partial_{z_1} 1) = X(\partial u) + \lambda_1 X(u) = X(\partial u) - \partial X(u),
\]
i.e., $\ChEnd_V(1) = \End_H(V)$.

In particular, we will denote by $I$ the identity operator on $V$, considered as an element of $\ChEnd_V(1)$.

Let us also consider the important case $n=2$. Any $X \in \ChEnd_V(2)$ is determined by a collection of bilinear maps $\alpha_j: V \otimes V \to V$, $j \ge -1$:
\[
X_{\lambda_1,\lambda_2}^{z_1, z_2} (u,v;\omega_2) = \alpha_{-1}(u,v) + \alpha_0(u,v) \lambda_1 + \alpha_1(u,v) \lambda_1^{2} + \cdots ,
\]
where the sum contains finitely many terms. It is convenient to write this sum as
\begin{equation}\label{eq:ChEnd(2)presentation}
X_{\lambda_1,\lambda_2}^{z_1, z_2} (u,v;\omega_2) = u \cdot_X v + \int\limits_{0}^{\lambda_1} (u \oo\sigma v)_{X}\, d\sigma,
\end{equation}
where $u \cdot_X v = \alpha_{-1}(u,v)$, and the remaining terms constitute a formal antiderivative of the polynomial $(u \oo\sigma v )_X \in \Bbbk [\sigma] \otimes V$, which, in turn, can be written as
\begin{equation}\label{eq:n-prod-lambda}
    (u \oo\sigma v )_X = \sum\limits_{n \ge 0} \dfrac{\sigma^n}{n!} (u \oo{(n)} v)_X, \quad (u \oo{(n)} v)_X = (n+1)! \alpha_n(u,v).
\end{equation}
The bilinear maps $(\cdot \oo{(n)} \cdot)_X$, $n \ge 0$, form a countable family of ordinary binary algebraic operations on $V$. It is also convenient to set $u \cdot_X v = (u \oo{(-1)} v)_X$.

The second equality in \eqref{eq:ChEnd-3/2-lin} implies that
\begin{equation}\label{eq:ConfBracket-1}
(u \oo{\lambda_1} v)_X = X^{z_1,z_2}_{\lambda_1,\lambda_2} (u,v; 1),
\end{equation}
since $1 = z_{21}\omega_2$ and the derivative with respect to $\lambda_2$ in the representation \eqref{eq:ChEnd(2)presentation} is zero. Then, from the first equality in \eqref{eq:ChEnd-3/2-lin} with $f=1$, it follows that
\[
((\partial + \lambda_1)u \oo{\lambda_1} v)_X = 0,
\]
i.e., the bracket $(\cdot \oo\sigma \cdot)_X$ satisfies the conformal sesquilinearity identity according to \cite{Kac_VA_Beginn}.

Note that the $\mathcal{D}_n$-linearity of the function $X$ implies that $\partial$ is a derivation with respect to the multiplication $\cdot_X$ and the bracket $(\cdot \oo\sigma \cdot)_X$. Indeed, the function $\omega_2$ obviously satisfies the equality
\[
\partial_{z_1}\omega_2 = -\partial_{z_2}\omega_2 = -\dfrac{1}{(z_2-z_1)^2}.
\]
Therefore,
\[
X_{\lambda_1,\lambda_2}^{z_1,z_2}((\partial+\lambda_1)u,v; \omega_2)
=
X_{\lambda_1,\lambda_2}^{z_1,z_2}(u,v; \partial_{z_1}\omega_2)
=
- X_{\lambda_1,\lambda_2}^{z_1,z_2}(u,v; \partial_{z_2}\omega_2)
=
-X_{\lambda_1,\lambda_2}^{z_1,z_2}(u,(\partial+\lambda_2)v; \omega_2).
\]
The left-hand side of this equality can be represented in the form \eqref{eq:ChEnd(2)presentation} as:
\begin{equation}\label{eq:Der3/2-left}
\lambda_1 (u \cdot_X v) + (\partial u \cdot_X v)
+\lambda_1 \int\limits_0^{\lambda_1} (u \oo\sigma v)_X d\sigma
+ \int\limits_0^{\lambda_1} (\partial u \oo\sigma v)_X d\sigma.
\end{equation}
The right-hand side can be written as
\begin{multline}\label{eq:Der3/2-right}
-\lambda_2 X_{\lambda_1,\lambda_2}^{z_1,z_2}(u,v;\omega_2)
- X_{\lambda_1,\lambda_2}^{z_1,z_2}(u,\partial v;\omega_2)
=
(\partial+\lambda_1)X_{\lambda_1,\lambda_2}^{z_1,z_2}(u,v;\omega_2)
- X_{\lambda_1,\lambda_2}^{z_1,z_2}(u,\partial v;\omega_2)
\\
=
(\partial+\lambda_1)(u \cdot_X v) +
(\partial+\lambda_1)\int\limits_0^{\lambda_1} (u \oo\sigma v)_X d\sigma
-(u \cdot_X \partial v)
-\int\limits_0^{\lambda_1} (u \oo\sigma \partial v)_X d\sigma.
\end{multline}
Comparing \eqref{eq:Der3/2-left} and \eqref{eq:Der3/2-right} at $\lambda_1=0$, we obtain
\[
(\partial u \cdot_X v) = \partial(u \cdot_X v) - (u \cdot_X \partial v),
\]
while the remaining terms yield the equality
\[
(\partial u \oo\sigma v)_X
=
\partial (u \oo\sigma v)_X - (u \oo\sigma \partial v)_X.
\]

Similarly to "ordinary" multilinear maps on a vector space, the chiral endomorphisms of an $H$-module $V$ form compositions. Specifically, for $X \in \ChEnd_V(n)$ and $Y_i \in \ChEnd_V(m_i)$, where $i=1,\dots, n$ and $m_1, \dots, m_n \geq 1$, we define their composition
\[
Z = X \circ (Y_1, \dots, Y_n) \in \ChEnd_V(m_1 + \dots + m_n)
\]
as follows. Let
$M_i = \sum\limits_{j=1}^{i} m_j$ and $\Lambda_i = \sum\limits_{j=M_{i-1}+1}^{M_i} \lambda_j$.
In particular, $M_0=0$ and $M=M_n=m_1+\dots+m_n$.

Since the ring $\Bbbk [z_1,\dots, z_M]$ is a unique factorization domain, the numerator of any fraction $f \in \mathcal{O}_M$ can be uniquely factored into $p_1 \dots p_n q$, where $p_i$ contains all irreducible factors of the numerator that depend only on $z_{lk}$ for $M_{i-1}+1 \le k < l \le M_i$, and the polynomial $q$ contains all "mixed" factors. By grouping all divisors in the denominator accordingly, one can represent $f$ in the form
\begin{equation}\label{eq:split-O}
f = g \cdot \prod_{i=1}^n f_i,
\end{equation}
where $f_i = f_i(z_{M_{i-1}+1}, \dots, z_{M_i})$ and $g \in \mathcal{O}_M$ has no poles at the points $z_{lk}$ for any $M_{i-1}+1 \le k < l \le M_i$.

Suppose $v_1,\dots, v_M \in V$ and
\[
(Y_i)_{\lambda_1,\dots, \lambda_{m_i}}^{z_1,\dots, z_{m_i}}(v_{M_{i-1}+1},\dots, v_{M_i}; f_i(z_1,\dots, z_{m_i}))
=\sum\limits_j F_{ij}(\lambda_1,\dots , \lambda_{m_i})\otimes_H w_{ij} \in \Lalg_{m_i}\otimes_H V
\]
for some $f_i \in \mathcal{O}_{m_i}$. Let us denote by $F$ the product of differential operators with coefficients in $\Lalg_M$:
\[
F_{j_1,\dots, j_n} = \prod\limits_{i=1}^n
    F_{ij_i}(\lambda_{M_{i-1}+1} - \partial_{z_{M_{i-1}+1}}, \ldots , \lambda_{M_{i}} - \partial_{z_{M_i}}).
\]

Then
\begin{equation} \label{eq:ChEnd-Comp}
    Z_{\lambda_1, \dots , \lambda_M}^{z_1, \dots , z_M}(v_1, \dots , v_M; f)
    =
    \sum\limits_{j_1,\dots, j_n}
    X^{z_{M_1}, \ldots , z_{M_n}}_{\Lambda_1, \ldots , \Lambda_n}
    \left( w_{1j_1},\dots, w_{nj_n};
    F_{j_1,\dots, j_n}g
    \big|_{z_k=z_{M_i},\, k=M_{i-1}+1,\dots , M_i-1}\right).
\end{equation}

Let us consider the simplest examples of compositions of chiral endomorphisms.

\begin{example} \label{exmp:ChEnd-Ex_1}
 Let us compute the composition $X \circ_1 Y = X \circ (Y, I)$ for $X, Y \in \ChEnd_V(2)$. We use the rule \eqref{eq:ChEnd-Comp}: here $n=2$, $m_1=2$, $m_2=1$, $M_1=2$, $M_2=M=3$, $\Lambda_1 = \lambda_1 + \lambda_2$, and $\Lambda_2 = \lambda_3$. The decomposition \eqref{eq:split-O} for $\omega_3$ takes the form
\[
\omega_3 = \dfrac{1}{z_{21}z_{31}z_{32}} = f_1 \cdot f_2 \cdot g = \dfrac{1}{z_{21}} \cdot 1 \cdot \dfrac{1}{z_{32}z_{31}}.
\]
Since $f_1 = \omega_2$, we have
\[
Y^{z_1,z_2}_{\lambda_1,\lambda_2}(a,b; f_1) = \sum_{n \ge -1} \dfrac{\lambda_1^{n+1}}{(n+1)!} (a \oo{(n)} b)_Y
\]
for $a, b \in V$.

Then for $Z = X \circ_1 Y$ and $a, b, c \in V$, we have
\begin{equation}     \label{eq:ChEnd_X1Y}
        Z_{\lambda_1,\lambda_2,\lambda_3}^{z_1,z_2,z_3}(a,b,c;\omega_3) = \sum_{n \ge -1}
        X_{\lambda_1+\lambda_2,\lambda_3}^{z_2, z_3}
        \left(
        (a \oo{(n)} b)_Y, c;
        \dfrac{(\lambda_1-\partial_{z_1})^{n+1}}{(n+1)!}
        \frac{1}{(z_3-z_2)(z_3-z_1)} \bigg|_{z_1=z_2} \right).
 \end{equation}

Note that for any polynomial $F(t) \in \Bbbk[t]$, the following equality holds:
\begin{equation} \label{eq:ChEnd-UsefulEq_1}
    F(\lambda_1 - \partial_{z_1}) \frac{1}{(z_3-z_2)(z_3-z_1)}\bigg|_{z_1=z_2} = \int\limits^{\partial_{z_2}}_{0} F(\lambda_1 - \tau) \,d\tau \left( \frac{1}{z_3-z_2}\right).
\end{equation}
Indeed, let us consider the left-hand side of \eqref{eq:ChEnd-UsefulEq_1} and apply Taylor's formula to the polynomial $F(\lambda_1-\partial_{z_1})$:
\begin{multline*}
    F(\lambda_1 - \partial_{z_1}) \frac{1}{(z_3-z_2)(z_3-z_1)}\bigg|_{z_1=z_2} = \sum_{n \geq 0} (-1)^n F^{(n)}(\lambda_1) \frac{\partial_{z_1}^n}{n!} \frac{1}{(z_3-z_2)(z_3-z_1)}\bigg|_{z_1=z_2} = \\ =  \sum_{n \geq 0} (-1)^n F^{(n)}(\lambda_1) \frac{1}{(z_3-z_2)^{n+2}}.
\end{multline*}
On the other hand, the right-hand side of \eqref{eq:ChEnd-UsefulEq_1} can be computed explicitly:
\begin{multline*}
    \int\limits^{\partial_{z_2}}_{0} F(\lambda_1 - \tau)\, d\tau \left(\frac{1}{z_3-z_2} \right)
    = \sum_{n \geq 0} (-1)^n F^{(n)}(\lambda_1)
    \int\limits^{\partial_{z_2}}_{0}  \frac{\tau^{n}}{n!}\,d\tau \left(\frac{1}{z_3-z_2}\right) = \\
    = \sum_{n \geq 0} (-1)^n F^{(n)}(\lambda_1) \frac{\partial^{n+1}_{z_2}}{(n+1)!} \frac{1}{z_3-z_2} = \sum_{n \geq 0} (-1)^n F^{(n)}(\lambda_1) \frac{1}{(z_3-z_2)^{n+2}}.
\end{multline*}
We see that the resulting expressions coincide.

Applying the equality \eqref{eq:ChEnd-UsefulEq_1} and the $\mathcal{D}_n$-linearity condition \eqref{eq:ChEnd-3/2-lin} to the right-hand side of \eqref{eq:ChEnd_X1Y}, we obtain
\begin{multline*}
        (X\circ_1 Y)_{\lambda_1,\lambda_2,\lambda_3}^{z_1,z_2,z_3}
        (a,b,c;\omega_3)
        =
        \sum_{n\ge -1}
        X_{\lambda_1+\lambda_2,\lambda_3}^{z_2,z_3}
        \left( (a\oo{(n)} b)_Y, c;
        \int_{0}^{\partial_{z_2}}
          \dfrac{(\lambda_1-\tau)^{n+1}}{(n+1)!}\,d\tau \dfrac{1}{z_3-z_2}
        \right )
        \\
        =
        \sum_{n\ge -1}
        X_{\lambda_1+\lambda_2,\lambda_3}^{z_2,z_3}
        \left( \int_{0}^{\partial+\lambda_1+\lambda_2}
        \dfrac{(\lambda_1-\tau)^{n+1}}{(n+1)!}\,d\tau
                ( a\oo{(n)} b)_Y, c ;\omega_2(z_2,z_3) \right).
 \end{multline*}
Note that the "upper" variables of the function $X$ are $z_2$ and $z_3$; therefore, by $\mathcal{D}_n$-linearity, the differentiation $\partial_{z_2}$ is transferred to the first argument of $X$ as the operator $\partial + \lambda_1 + \lambda_2$, since the "lower" variables of $X$ are $\lambda_1 + \lambda_2$ and $\lambda_3$. Furthermore, $\dfrac{1}{z_3-z_2}$ coincides with $\omega_2$ in the variables $z_2, z_3$.

We conclude the computation of the composition $X \circ_1 Y$:
\begin{multline}\label{eq:X-1-Yformula}
        (X\circ_1 Y)_{\lambda_1,\lambda_2,\lambda_3}^{z_1,z_2,z_3}
        (a,b,c;\omega_3) \\
       =
X_{\lambda_1+\lambda_2,\lambda_3}^{z_2,z_3}((\partial+\lambda_1+\lambda_2)(a\cdot_Y b), c; \omega_2)
+
\sum_{n\ge 0}
        X_{\lambda_1+\lambda_2,\lambda_3}^{z_2,z_3}
        \left( \int_{0}^{\partial+\lambda_1+\lambda_2}
        \dfrac{(\lambda_1-\tau)^{n+1}}{(n+1)!}\,d\tau
                ( a\oo{(n)} b)_Y, c ; \omega_2 \right)
\\
=
X_{\lambda_1+\lambda_2,\lambda_3}^{z_2,z_3}((\partial+\lambda_1+\lambda_2)(a\cdot_Y b), c; \omega_2)
+
\sum_{n\ge 0}
        X_{\lambda_1+\lambda_2,\lambda_3}^{z_2,z_3}
        \left( \dfrac{\lambda_1^{n+2}}{(n+2)!}
                ( a\oo{(n)} b)_Y, c ; \omega_2 \right)
\\
-
\sum_{n\ge 0}
        X_{\lambda_1+\lambda_2,\lambda_3}^{z_2,z_3}
        \left( \dfrac{(-\partial-\lambda_2)^{n+2}}{(n+2)!}
                ( a\oo{(n)} b)_Y, c ; \omega_2 \right).
\end{multline}
\end{example}

\begin{example} \label{exmp:ChEnd-Ex_2}
Let us compute the composition $X \circ_2 Y = X \circ (I, Y)$ for $X, Y \in \ChEnd_V(2)$. We use the composition rule \eqref{eq:ChEnd-Comp}: in this case $n = 2$, $m_1 = 1$, $m_2 = 2$, $M_1 = 1$, $M_2 = 3$, $\Lambda_1 = \lambda_1$, and $\Lambda_2 = \lambda_2 + \lambda_3$.
The decomposition \eqref{eq:split-O} for $\omega_3$ takes the form
\[
\omega_3 = \frac{1}{z_{21}z_{31}z_{32}} = f_1 \cdot f_2 \cdot g = 1 \cdot \frac{1}{z_{32}} \cdot \frac{1}{z_{21}z_{31}}.
\]
Since $f_2 = \omega_2(z_2, z_3)$, we have
\[
Y_{\lambda_2, \lambda_3}^{z_2,z_3}(b,c;f_2) = \sum_{n \ge -1} \frac{\lambda_2^{n+1}}{(n+1)!} (b \oo{(n)} c)_Y
\]
for $b, c \in V$.

Then for $Z = X \circ_2 Y$ and $a, b, c \in V$, we have
\begin{multline}
Z^{z_1,z_2,z_3}_{\lambda_1,\lambda_2,\lambda_3}(a,b,c;\omega_3) = \sum_{n \geq -1} X_{\lambda_1,\lambda_2+\lambda_3}^{z_1,z_3} \left(a, (b \oo{(n)} c)_{Y} ; \frac{(\lambda_2 - \partial_{z_2})^{n+1}}{(n+1)!} \frac{1}{(z_2-z_1)(z_3-z_1)} \bigg|_{z_2=z_3} \right).
\end{multline}
Next, we use equality \eqref{eq:ChEnd-UsefulEq_1} with permuted variables:
\begin{equation}\label{eq:ChEnd-UsefulEq_2}
    F(\lambda_2-\partial_{z_2}) \frac{1}{(z_2-z_1)(z_3-z_1)}\bigg|_{z_2=z_3} = - \int\limits_{0}^{\partial_{z_3}} F(\lambda_2-\tau) d\tau \left( \frac{1}{z_3-z_1} \right).
\end{equation}
Applying \eqref{eq:ChEnd-UsefulEq_2} and the $\mathcal{D}_n$-linearity condition, we obtain
\begin{multline}
    (X \circ_2 Y)^{z_1,z_2,z_3}_{\lambda_1,\lambda_2,\lambda_3}(a,b,c;\omega_3) = \sum_{n \geq -1} X_{\lambda_1,\lambda_2+\lambda_3}^{z_1,z_3} \left(a, (b \oo{(n)} c)_Y ; -\int\limits_{0}^{\partial_{z_3}} \frac{(\lambda_2 - \tau)^{n+1}}{(n+1)!} \, d\tau \frac{1}{z_3-z_1} \right) \\
    = \sum_{n \ge -1} X_{\lambda_1,\lambda_2+\lambda_3}^{z_1,z_3} \left( a, - \int\limits_{0}^{\partial+\lambda_2+\lambda_3} \frac{(\lambda_2 - \tau)^{n+1}}{(n+1)!} \, d\tau (b \oo{(n)} c)_Y ; \omega_2(z_1,z_3) \right).
\end{multline}
\end{example}

Note that by Taylor's formula,
\[
\int\limits_0^{\partial+\lambda_2+\lambda_3} F(\tau)\,d\tau
= \int\limits_0^{\lambda_2+\lambda_3} F(\tau)\,d\tau
+ \partial F(\lambda_2+\lambda_3) +
\dfrac{\partial^2}{2!} F'(\lambda_2+\lambda_3) + \dots ,
\]
which for $F(\tau ) = \dfrac{(\lambda_2-\tau)^{n+1}}{(n+1)!}$ yields
\begin{multline*}
\int\limits_0^{\partial+\lambda_2+\lambda_3} \dfrac{(\lambda_2-\tau)^{n+1}}{(n+1)!}\,d\tau
=
\int\limits_0^{\lambda_2+\lambda_3} \dfrac{(\lambda_2-\tau)^{n+1}}{(n+1)!}\,d\tau
+\sum\limits_{s= 1}^{n+2}(-1)^{s-1}
\dfrac{\partial^s}{s!} \dfrac{(-\lambda_3)^{n+2-s}}{(n+2-s)!}\\
=\dfrac{\lambda_2^{n+2}}{(n+2)!}
+
\sum\limits_{s= 0}^{n+2}(-1)^{s-1}
\dfrac{\partial^s}{s!} \dfrac{(-\lambda_3)^{n+2-s}}{(n+2-s)!}
=
\dfrac{\lambda_2^{n+2}}{(n+2)!}
-
\sum\limits_{s= 0}^{n+2}
\dfrac{(-\partial)^s}{s!} \dfrac{(-\lambda_3)^{n+2-s}}{(n+2-s)!}.
\end{multline*}

Returning to the computation of the composition $X \circ_2 Y$:
\begin{multline*}
(X \circ_2 Y)_{\lambda_1,\lambda_2,\lambda_3}^{z_1,z_2,z_3}(a,b,c;\omega_3) \\
=
\sum_{n \ge -1} X_{\lambda_1,\lambda_2+\lambda_3}^{z_1,z_3} \left( a, - \dfrac{\lambda_2^{n+2}}{(n+2)!}(b \oo{(n)} c)_Y
+
\sum_{s\ge 0}
\dfrac{(-\partial)^s}{s!} \dfrac{(-\lambda_3)^{n+2-s}}{(n+2-s)!} (b \oo{(n)} c)_Y ; \omega_2 \right)
\\
=
-\lambda_2 X_{\lambda_1,\lambda_2+\lambda_3}^{z_1,z_3} (a, (b \cdot_Y c); \omega_2)
-\sum_{n \ge 0}
\dfrac{\lambda_2^{n+2}}{(n+2)!}
X_{\lambda_1,\lambda_2+\lambda_3}^{z_1,z_3} (a, (b \oo{(n)} c)_Y; \omega_2)
\\
-\lambda_3 X_{\lambda_1,\lambda_2+\lambda_3}^{z_1,z_3}
\left(a, (b \cdot_Y c) ; \omega_2 \right)
-
X_{\lambda_1,\lambda_2+\lambda_3}^{z_1,z_3}
(a,\partial (b \cdot_Y c) ; \omega_2)
\\
+ \sum_{n\ge 0}
\sum_{s= 0}^{n+2}
\dfrac{(-\lambda_3)^{n+2-s}}{(n+2-s)!}
X_{\lambda_1,\lambda_2+\lambda_3}^{z_1,z_3}
\left(a,
\dfrac{(-\partial)^s}{s!} (b \oo{(n)} c)_Y ; \omega_2 \right).
\end{multline*}
In this expression, the explicit occurrences of the variable $\lambda_3$ should be replaced by $-\partial - \lambda_1 - \lambda_2$ in accordance with the structure of $L_3 \otimes_H V$. Using the differential property of $\partial$, we obtain
\begin{multline}\label{eq:X-2-Yformula}
(X \circ_2 Y)_{\lambda_1,\lambda_2,\lambda_3}^{z_1,z_2,z_3}(a,b,c;\omega_3)
=
\\
X_{\lambda_1,\lambda_2+\lambda_3}^{z_1,z_3}((\lambda_1+\partial)a, b \cdot_Y c; \omega_2)
-\sum_{n \ge 0}
\dfrac{\lambda_2^{n+2}}{(n+2)!}
X_{\lambda_1,\lambda_2+\lambda_3}^{z_1,z_3} (a, (b \oo{(n)} c)_Y; \omega_2)
\\
+ \sum_{n\ge 0}
\sum_{s= 0}^{n+2}
\dfrac{(\partial+\lambda_1+\lambda_2)^{n+2-s}}{(n+2-s)!}
X_{\lambda_1,\lambda_2+\lambda_3}^{z_1,z_3}
\left(a,
\dfrac{(-\partial)^s}{s!} (b \oo{(n)} c)_Y ; \omega_2 \right).
\end{multline}

We define the right action of the symmetric group $S_n$ on $\ChEnd_V(n)$ as follows: for $X \in \ChEnd_V(n)$ and $\sigma \in S_n$, let
\begin{equation}\label{eq:ChEnd-SymAct}
    (X^{\sigma})_{\lambda_1, \dots , \lambda_n}^{z_1, \dots , z_n}(v_1, \dots , v_n; f) = X_{\lambda_{\sigma^{-1}(1)}, \dots , \lambda_{\sigma^{-1}(n)}}^{z_1 , \dots , z_n}(v_{\sigma^{-1}(1)}, \dots , v_{\sigma^{-1}(n)}; f^{\sigma}),
\end{equation}
where $v_i \in V$, $f \in \mathcal{O}_n$, and $f^\sigma = f(z_{\sigma^{-1}(1)}, \dots, z_{\sigma^{-1}(n)})$. In other words, the group $S_n$ acts on $\ChEnd_V(n)$ from the right by permuting the elements $v_1, \dots , v_n$ and simultaneously permuting the variables $\lambda_1, \dots , \lambda_n$.

\begin{example}\label{exmp:(12)Action}
Let $X \in \ChEnd_V(2)$ be given by formula \eqref{eq:ChEnd(2)presentation}. Then
\[
(X^{(12)})^{z_1,z_2}_{\lambda_1,\lambda_2}
(u,v; \omega_2)
=
X^{z_1,z_2}_{\lambda_2,\lambda_1} (v,u; \omega_2^{(12)})
= - v\cdot_X u - \int_{0}^{\lambda_2}
(v\oo\sigma u)_X d\sigma
= - v\cdot_X u - \!\!\int_{0}^{-\partial-\lambda_1}
(v\oo\sigma u)_X d\sigma ,
\]
since multiplication by $\lambda_2 + \lambda_1$ coincides with the action of $\partial$.
Let us decompose the last integral on the right-hand side into the sum
\[
\int_{0}^{-\partial-\lambda_1}
(v\oo\sigma u)_X d\sigma
=
\int_{0}^{-\partial}
(v\oo\sigma u)_X d\sigma
+
\int_{-\partial }^{-\partial-\lambda_1}
(v\oo\sigma u)_X d\sigma
\]
and perform a change of variable $\sigma = -\partial - \tau$ in the second summand. We obtain
\[
(X^{(12)})^{z_1,z_2}_{\lambda_1,\lambda_2}
(u,v; \omega_2)
=
-v\cdot_X u + \int_{-\partial}^{\lambda_1}
(v\oo{-\partial-\tau } u)_X d\tau
= - v\cdot_X u + \int_{-\partial}^0 (v \oo{\tau} u)_X d\tau
+ \int_{0}^{\lambda_1} (v\oo{-\partial-\tau} u)_X d\tau.
\]
Thus, in particular,
\begin{equation}\label{eq:Commutator}
u\cdot_{X^{(12)}} v = -v\cdot_X u +
\int_{-\partial}^0 (v \oo{\tau} u)_X d\tau,
\quad
(u\oo{\sigma} v)_{X^{(12)}}
 = (v\oo{-\partial-\sigma } u)_X.
\end{equation}
\end{example}

\begin{proposition}[\cite{BDKH-2019}]
The family of spaces $\ChEnd_V(n)$, $n \ge 1$, equipped with the composition \eqref{eq:ChEnd-Comp} and the action of the symmetric groups \eqref{eq:ChEnd-SymAct}, is a symmetric operad.
\end{proposition}

\begin{example}
The equivariance property of composition in a symmetric operad implies, in particular,
\[
  X\circ_2 Y = (X^{(12)}\circ( Y^{(12)}, I))^{(13)}
  = (X^{(12)}\circ_1 Y^{(12)})^{(13)}.
\]
Indeed, let us consider the summands containing both operations $\cdot_X$ and $\cdot_Y$ in the formula \eqref{eq:X-2-Yformula} for $X\circ_2 Y$ evaluated on elements $a,b,c\in V$:
\[
(\lambda_1+\partial )a\cdot_X(b\cdot_Y c).
\]
On the other hand, computing $(X^{(12)}\circ_1 Y^{(12)})^{(13)}$ on the elements $a,b,c\in V$ using formula \eqref{eq:X-1-Yformula} and taking \eqref{eq:Commutator} into account, we obtain the following summands with the operations $\cdot_X$ and $\cdot_Y$:
\[
- a\cdot_X (\partial+\lambda_2+\lambda_3)(b\cdot_Y c)
\]
(the minus sign appears because $\omega_3^{(13)}=-\omega_3$). Eliminating $\lambda_3$ as was done in Example~\ref{exmp:ChEnd-Ex_2}, we see that the resulting expressions coincide:
\begin{multline*}
- a\cdot_X (\partial+\lambda_2+\lambda_3)(b\cdot_Y c)
= -(\lambda_2+\lambda_3) (a\cdot_X (b\cdot_Y c))
-(a\cdot _X \partial (b\cdot_Y c))
\\
=(\lambda_1+\partial) (a\cdot_X (b\cdot_Y c))
-(a\cdot _X \partial (b\cdot_Y c))
=((\lambda_1+\partial)a\cdot _X (b\cdot_Y c)).
\end{multline*}
\end{example}

\begin{proposition}\label{prop:HadamardEnd}
Let $V$ be an $H$-module, and let $P$ be a vector space. Then there exists a morphism of operads
\[
\End_P \otimes \ChEnd_V \to \ChEnd_{P\otimes V},
\]
where $P\otimes V$ is viewed as an $H$-module with the action $\partial(p\otimes u) = p\otimes \partial u$.
\end{proposition}

\begin{proof}
For $t\in \End_P(n)$ and $X\in \ChEnd_V(n)$, we define the image of $t\otimes X$ as the map
\begin{equation}\label{eq:HadamardEnd}
(t\otimes X)^{z_1,\dots, z_n}_{\lambda_1,\dots,\lambda_n}
(p_1\otimes u_1, \dots , p_n\otimes u_n; f)
=
t(p_1,\dots, p_n)\otimes X^{z_1,\dots, z_n}_{\lambda_1,\dots,\lambda_n}
(u_1, \dots , u_n; f)
\end{equation}
for $p_i\in P$, $u_i\in V$, and $f\in \mathcal O_n$.

The constructed map $(t\otimes X): (P\otimes V)^{\otimes n}\otimes \mathcal O_n \to L_n\otimes _H (P\otimes V)$ is $\mathcal D_n$-linear because $\partial$ acts on the second tensor factor. Since the compositions and the action of $S_n$ on the Hadamard product $\End_P\otimes \ChEnd_V$ are defined in the componentwise way,  the rule \eqref{eq:HadamardEnd} is compatible with the operad structure.
\end{proof}

\section{Manin product and left-symmetric chiral algebras}

Let $\Var$ be a variety of algebras with arbitrary operations (not necessarily binary) satisfying a family of multilinear identities. We will also denote the operad corresponding to this variety by $\Var$. Then the structure of a $\Var$-algebra on a vector space $A$ is exactly a morphism of operads $\Var \to \End_A$, where $\End_A$ is the operad of multilinear maps on $A$.

By replacing the operad of multilinear maps with the operad $\ChEnd_V$ for an $H$-module $V$, we obtain the definition of a {\em $\Var$-chiral algebra}: this is a morphism
\[
\Var \to \ChEnd_V.
\]

As shown in \cite{BDKH-2019}, the difference between the definition of a Lie chiral algebra and the classical definition of a vertex algebra lies solely in the fact that the morphism $\Lie \to \ChEnd_V$ does not explain what is the vacuum vector.

In this section, we will derive the relations of a left-symmetric chiral algebra in terms of the operations \eqref{eq:ChEnd(2)presentation}. To do so, we use the fact that the operad $\LSym$ of left-symmetric algebras is the result of a dendriform splitting of the operad $\Lie$ of Lie algebras. In general, the dendriform splitting construction is described by the so-called Manin product of operads.

The Manin products (white and black) for binary quadratic operads were introduced in \cite{GK-94}. The computation of the defining relations of these products for each specific pair of operads is a routine linear algebra problem. For the  Manin black product, this set of defining relations can be obtained using the following procedure \cite{KolSart2025_JAlg}.

Let $\Var$ and $\mathcal P$ be binary quadratic operads generated by finite-dimensional $S_2$-spaces $U$ and $E$, respectively. The operad $\Var$ is the image of the free operad $\mathcal F_U$ generated by the $S_2$-space $U$. Let us choose a basis $(\mu_j)_{j\in J}$ of the space $U$ and a basis $(e_i)_{i\in I}$ of the space $E$. The dual operad $\mathcal P^!$ is generated by the space of linear functions $E^*$, on which the action of the permutation $(12)\in S_2$ is given by the rule $\langle \alpha^{(12)}, e\rangle = -\langle \alpha , e^{(12)}\rangle$ for $\alpha \in E^*$ and $e\in E$. We denote the dual basis of the space $E^*$ by $(e_i^*)_{i\in I}$.

Let us denote by $\mathcal F_{M}$ the free operad generated by the $S_2$-space $M = U\otimes E\otimes \Bbbk_-$ with the componentwise action of the group $S_2$. Here and below, $\Bbbk_-$ denotes the 1-dimensional $S_2$-module with the skew-symmetric action. A basis of the space $M$ is formed by the tensors $u_{ij} = \mu_j\otimes e_i\otimes 1$ for $i\in I$ and $j\in J$.

Let us denote by $\Phi$ the  morphism of operads $\mathcal F_U \to \mathcal P^!\otimes \mathcal F_{M}$ given by the rule
\begin{equation}\label{eq:Phi-morphism}
\Phi (\mu_j) = \sum\limits_{i\in I} e_i^* \otimes u_{ij},\quad j\in J.
\end{equation}

As shown in \cite{KolSart2025_JAlg}, a morphism $\Psi :\mathcal F_{M}\to \mathcal Q$, where $\mathcal Q$ is some operad, defines the black Manin product of the operads $\Var$ and $\mathcal P$ if the kernel of the composition
\[
\Phi\circ (\mathrm{id}\otimes \Psi):    \mathcal F_U \to \mathcal P^!\otimes \mathcal F_{M}
    \to \mathcal P^!\otimes \mathcal Q
\]
coincides with the kernel of the epimorphism $\mathcal F_U\to \Var$. In this case, $\mathcal Q$ is denoted by $\Var\bullet \mathcal P$.

In other words, to compute the defining relations of the operad $\Var\bullet \mathcal P$, it is necessary to compute the images under $\Phi$ of all defining relations of the operad $\Var$.

\begin{example}\label{exmp:Pois}
Let $\mathcal P=\Com$ be the operad of commutative algebras; then $\mathcal P^!=\Lie$ is the operad of Lie algebras. Let us consider the operad $\Pois$ of Poisson algebras as $\Var$. The depolarized presentation of this operad (see, e.g., \cite{MarklRemm2004}) in terms of a single binary operation $\mu = x_1\cdot x_2$ contains a single relation
\begin{equation}\label{eq:Poisson-one}
p = x_1 \cdot (x_2 \cdot x_3) - (x_1 \cdot x_2) \cdot x_3
+ \dfrac{1}{3}
((x_1 \cdot x_3) \cdot x_2 + (x_2 \cdot x_3) \cdot x_1 - (x_2 \cdot x_1) \cdot x_3 - (x_3 \cdot x_1) \cdot x_2) .
\end{equation}
Thus, in our case $E=\Bbbk e$, $e^{(12)}=e$, and $U=\Bbbk \mu + \Bbbk \mu^{(12)}$. The morphism $\Phi$ maps the generating operation $\mu = x_1\cdot x_2$ of the free operad to the tensor $[x_1,x_2]\otimes x_1x_2$, where $x_1x_2 = \mu\otimes e\otimes 1$. Consequently, the relation \eqref{eq:Poisson-one} is mapped to
\begin{multline*}
\Phi(p) = [x_1,[x_2,x_3]]\otimes x_1 (x_2x_3)
-
[[x_1,x_2],x_3]\otimes (x_1 x_2) x_3 \\
+ \dfrac{1}{3}
(
[[x_1,x_3],x_2]\otimes (x_1 x_3) x_2
+[[x_2,x_3],x_1]\otimes  (x_2 x_3) x_1
\\
-[[x_2,x_1],x_3]\otimes  (x_2 x_1) x_3
-[[x_3,x_1],x_2]\otimes  (x_3 x_1) x_2
) .
\end{multline*}
Let us group the summands with the same first tensor factors from the space $\Lie(3)$ using the basis $e_1 = [[x_1,x_2],x_3]$ and $e_2 = [[x_1,x_3],x_2]$: the Jacobi identity allows us to write
\[
[x_1,[x_2,x_3]] = e_1-e_2.
\]
Hence, we obtain two defining relations of the operad $\Pois\bullet \Com$:
\[
r_1 = x_1 (x_2x_3)-(x_1 x_2) x_3
+ \dfrac{1}{3}
(-(x_2 x_3) x_1+(x_2 x_1) x_3),
\]
\[
r_2= -x_1 (x_2x_3)
+ \dfrac{1}{3}
((x_1 x_3) x_2+(x_2 x_3) x_1+(x_3 x_1) x_2).
\]
The relations $r_1$ and $r_2$ generate an $S_3$-space, which is easy to compute. Consider the sum
\[
3(r_1+r_2) = -2(x_1x_2)x_3 + (x_3x_1)x_2 + (x_2x_1)x_3.
\]
Together with the relations obtained by all possible permutations of the variables, this sum generates the space of polynomials containing $(x_1x_2)x_3-(x_{\sigma(1)}x_{\sigma(2)})x_{\sigma(3)}$ for $\sigma \in S_3$. The relation $r_2$ then turns into the usual associativity.

Thus, the operad $\Pois\bullet\Com$ is generated by a single non-symmetric binary operation, and the defining relations state that all monomials of the degree 3 are equal to each other, i.e., this operad is dual to the operad of Lie-admissible algebras.
\end{example}

Let $\mathcal P = \LSym$; then $\mathcal P^! = \Perm$ is the operad of associative left-commutative algebras, i.e., the following identities hold in $\mathcal P^!$-algebras:
\[
(x_1x_2)x_3 = x_1(x_2x_3)=x_2(x_1x_3).
\]
In this case, the space $E$ is two-dimensional and is isomorphic to $\Bbbk S_2$ as an $S_2$-module. Let us choose the basis $e_1=x_1x_2$,  $e_2 = e_1^{(12)}=x_2x_1$ of the space $E$. Note that for the dual basis $e_1^*, e_2^*$, we have
\[
(e_1^*)^{(12)}= - e_2^*.
\]

Let us now consider as $\Var$ the operad of Lie algebras, generated by  one-dimensional space $U$ with the basis $\mu = [x_1x_2]$, $\mu^{(12)} = -\mu$. Then $M=U\otimes E\otimes \Bbbk_-$ is a two-dimensional space with a basis $x_1\succ x_2 = \mu\otimes e_1\otimes 1$, $x_1\prec x_2 = \mu\otimes e_2\otimes 1$. It follows directly from the definition that
\[
x_2\succ x_1 = \mu^{(12)}\otimes e_1^{(12)}\otimes 1^{(12)} =(-\mu)\otimes e_2\otimes (-1) = x_1\prec x_2.
\]

The morphism $\Phi :\mathcal F_U\to \Perm\otimes \mathcal F_M$ takes the form
\[
\Phi(\mu )=e_1^*\otimes (x_1\succ x_2) - (e_1^*)^{(12)}\otimes (x_1\prec x_2),
\]
where $e_1^*$ and $e_2^*$ satisfy the relations of the $\Perm$ operad. Specifically, if we denote $e_1^* = p_1p_2$, then $e_2^* = -(e_1^*)^{(12)} = -p_2p_1$ and
\[
(p_1p_2)p_3 = p_1(p_2p_3) = p_2(p_1p_3).
\]

Thus, for any $\Perm$-algebra $P$ defined by a morphism $\Perm \to \End_P$, and for any chiral algebra $V$ whose structure is defined by a morphism $\mathcal F_M\to \ChEnd_V$, we obtain a chain of morphisms
\[
\mathcal F_U \to \Perm\otimes \mathcal F_M \to \End_P\otimes \ChEnd_V \to \ChEnd_{P\otimes V}.
\]
The last map is defined by Proposition \ref{prop:HadamardEnd}.

Consequently, the morphism $\Psi :\LSym \to \ChEnd_V$ defining the structure of a \emph{left-symmetric chiral algebra} on the $H$-module $V$ is given by operations $X,Y \in \ChEnd_V(2)$ such that $\Psi: x_1\succ x_2 \mapsto X$ and $x_1\prec x_2\mapsto Y$, where $Y = X^{(12)}$ and
\begin{equation}\label{eq:LSymPhiChiral}
(p\otimes a)\otimes (q\otimes b)\otimes f(z_1,z_2)
\mapsto
pq\otimes X^{z_1,z_2}_{\lambda_1,\lambda_2}(a,b; f)
- qp\otimes Y^{z_1,z_2}_{\lambda_1,\lambda_2}(a,b; f)
\end{equation}
determines the structure of a Lie chiral  algebra on the $H$-module $P\otimes V$ for any $\Perm$-algebra $P$. The described procedure is a particular case of the \emph{dendriform splitting} of an operad in the form described in \cite{GubKol-2014}.

Let us denote
\[
X^{z_1,z_2}_{\lambda_1,\lambda_2}(a,b; \omega_2)
= a\succ b + \int\limits_{0}^{\lambda_1} (a\oo{\sigma } b)\,d\sigma.
\]
Then, in view of \eqref{eq:Commutator}, the right-hand side of \eqref{eq:LSymPhiChiral} for $f=\omega_2$ takes the form
\[
pq\otimes (a\succ b) + qp\otimes (b\succ a)
- qp\otimes \int\limits_{-\partial}^0 (b\oo\tau a)\,d\tau
+ \int\limits_{0}^{\lambda _1} ( pq\otimes (a\oo\sigma b) -qp\otimes (b\oo{-\partial-\sigma } a))\,d\sigma .
\]
Thus, a left-symmetric chiral algebra is an $H$-module $V$ equipped with an operation $\succ : V\otimes V\to V$ and a 3/2-linear bracket $(\cdot \oo\sigma \cdot ): V\otimes V\to V[\sigma ]$ such that for any $\Perm$-algebra $P$ the space $P\otimes V$ with the operations
\begin{equation}\label{eq:SplitChiralOperations}
\begin{gathered}
(p\otimes a) \cdot (q\otimes b)
 = pq\otimes (a\succ b) + qp\otimes (b\succ a)
- qp\otimes \int\limits_{-\partial}^0 (b\oo\tau a)\,d\tau,
\\
[(p\otimes a) \oo\sigma (q\otimes b)] =  pq\otimes (a\oo\sigma b) -qp\otimes (b\oo{-\partial-\sigma } a)
\end{gathered}
\end{equation}
is a Lie chiral algebra. Recall \cite{BK2002} that the operations of a Lie chiral algebra satisfy the following axioms:

\begin{itemize}
    \item Differential properties derived from the $\mathcal D_2$-linearity of the operation in a chiral algebra:
    \[
    \partial (x\cdot y) = (\partial x\cdot y) + (x\cdot \partial y),
    \quad
    [\partial x\oo\sigma  y] = -\sigma [x\oo\sigma y],
    \quad
    [x\oo\sigma \partial y] = (\partial+\sigma )[x\oo\sigma y].
    \]
    \item Skew-symmetry:
    \[
    x\cdot y = y\cdot x +\int\limits_{-\partial }^0 [x\oo\sigma y]\,d\sigma ,
    \quad
    [x\oo\sigma y ] = - [y\oo{-\partial -\sigma } x];
    \]
    \item Associator identity:
    \[
    (x\cdot y)\cdot z - x\cdot (y\cdot z)
    =\sum\limits_{n\ge 0}
    \dfrac{1}{(n+1)!}\partial^{n+1}
    (x\cdot [y\oo{(n)} z] + y\cdot [x\oo{(n)} z]),
    \]
    where, as above, $[x\oo{(n)} y] $ denotes the coefficient of $\lambda^n/n!$ in the polynomial $[x\oo\lambda y]$.
    \item Wick identity:
    \[
    [x\oo\sigma (y\cdot z)] =
    [x\oo\sigma y]\cdot z  + y\cdot [x\oo\sigma z]
    +\int\limits_0^\sigma [[x\oo\sigma y]\oo\lambda z]\,d\lambda .
    \]
    \item Conformal Jacobi identity:
    \[
    [x\oo\lambda [y\oo\sigma z]] -[y\oo\sigma [x\oo\lambda  z]]  = [[x\oo\lambda y]\oo{\sigma +\lambda} z].
    \]
\end{itemize}

Substituting expressions \eqref{eq:SplitChiralOperations} for the free $\Perm$-algebra $P$ into these axioms, we derive the conditions on the operations $\succ$ and $(\cdot\oo\sigma \cdot)$ that define a left-symmetric chiral algebra.

\begin{itemize}
    \item Differential properties:
    \begin{equation}\label{eq:LSym-Chiral-1}
    (\partial a\oo\sigma  b) = -\sigma (a\oo\sigma b),
    \quad
    (a\oo\sigma \partial b) = (\partial+\sigma )(a\oo\sigma b),
    \quad
    \partial (a\succ b) = (\partial a\succ b) + (a\succ \partial b).
    \end{equation}
\end{itemize}
Note that skew-symmetry and the commutator identity for operations \eqref{eq:SplitChiralOperations} are satisfied automatically.
For ease of notation, we will use the following definitions:
\begin{equation}\label{eq:Notations-prec-succ}
a\prec b = b\succ a -\int\limits_{-\partial }^0 (b\oo\tau a)\,d\tau ,
\quad
\{a\oo{(n)} b\} = \sum\limits_{s\ge 0}\dfrac{(-1)^{n+s}}{s!}\partial ^s (a\oo{(n+s)} b).
\end{equation}
\begin{itemize}
    \item The associator identity splits into three relations:
\begin{gather}
    (a*b)\succ c - a\succ (b\succ c)
     = \sum\limits_{n\ge 0} \dfrac{1}{(n+1)!} \partial ^{n+1} (a\succ (b\oo{(n)} c) + b\succ (a\oo{(n)} c) ), \label{eq:LSym-Chiral-Z1} \\
     (a\succ b)\prec c - a\succ (b\prec c)
     =
     \sum\limits_{n\ge 0} \dfrac{1}{(n+1)!} \partial ^{n+1} (b\prec [a\oo{(n)} c] -a\succ \{c\oo{(n)}b\}),
      \label{eq:LSym-Chiral-Z2} \\
     (a\prec b)\prec c - a\prec (b*c)
     =
     \sum\limits_{n\ge 0} \dfrac{1}{(n+1)!} \partial ^{n+1} (a\prec [b\oo{(n)} c] - b\succ \{c\oo{(n)}a\} ),
       \label{eq:LSym-Chiral-Z3}
\end{gather}
where $a*b = a\succ b + a\prec b$ and $[a\oo{(n)} b] = (a\oo{n} b) - \{b\oo{(n)} a\}$.
    \item Splitting the Wick identity leads to the following conditions:
\begin{gather}
    (a\oo\sigma (b\succ c)) = [a\oo\sigma b]\succ c + b\succ (a\oo\sigma c) +\int\limits_0^\sigma ([a\oo\sigma b]\oo\lambda c)\,d\lambda , \label{eq:LSym-Chiral-W1} \\
    (a\oo\sigma (b\prec c)) = (a\oo\sigma b)\prec c + b\prec [a\oo\sigma c] -\int\limits_0^\sigma \{ c\oo\lambda (a\oo\sigma b)\}\,d\lambda ,
    \label{eq:LSym-Chiral-W2} \\
    \{(b*c)\oo\sigma a \} = \{b\oo\sigma a\}\prec c + b\succ \{a\oo\sigma a\} - \int\limits_0^\sigma \{c\oo\lambda \{b\oo\sigma a\}\}\,d\lambda .
    \label{eq:LSym-Chiral-W3}
\end{gather}
    \item The conformal Jacobi identity on $P\otimes V$ also yields three identities on $V$:
\begin{gather}
    a\oo\lambda (b\oo\sigma c) - b\oo\sigma (a\oo\lambda c) =
    ([a\oo\lambda b]\oo{\sigma +\lambda }c),
     \label{eq:LSym-Chiral-J1} \\
     a\oo\lambda \{c\oo\sigma b\} - \{[a\oo\lambda c]\oo\sigma b\} = \{c\oo{\sigma +\lambda } (a\oo\lambda b)\},
     \label{eq:LSym-Chiral-J2} \\
     \{ [b\oo\sigma c]\oo\lambda a\}
     + \{c\oo{\sigma +\lambda } \{b\oo\lambda a\}\}
     = b\oo\sigma \{ c\oo\lambda a\}.
     \label{eq:LSym-Chiral-J3}
\end{gather}
\end{itemize}

Note that the relations \eqref{eq:LSym-Chiral-J1}--\eqref{eq:LSym-Chiral-J3} are equivalent to the left-symmetry condition \cite{HongLi2015} of the conformal component of the algebra $V$. The identities \eqref{eq:LSym-Chiral-Z1}--\eqref{eq:LSym-Chiral-W3} are not independent: for example, one can derive \eqref{eq:LSym-Chiral-W2} from \eqref{eq:LSym-Chiral-W1} by rewriting it in terms of the $\succ$ operation using \eqref{eq:Notations-prec-succ}. Similarly, \eqref{eq:LSym-Chiral-Z1} implies \eqref{eq:LSym-Chiral-Z2} and \eqref{eq:LSym-Chiral-Z3}.

We also note two facts that are essential for the purposes of this work.

\begin{remark}\label{rem:LSYmAb-Zinb}
If in a left-symmetric chiral algebra $V$ the equality $(a\oo\lambda b)=0$ holds for all $a,b\in V$, then $(V,\succ )$ is a pre-commutative (Zinbiel) algebra with a derivation $\partial$.
\end{remark}

\begin{proposition}\label{prop:QuadriLie}
Let $V$ be a left-symmetric chiral algebra. Then $(V,\succ, \prec )$ satisfies the axioms of a pre-left-symmetric (L-dendriform) algebra.
\end{proposition}

\begin{proof}
    It suffices to note that the right-hand sides of \eqref{eq:LSym-Chiral-Z2} and \eqref{eq:LSym-Chiral-Z3} coincide upon swapping the variables $a$ and $b$. Consequently, for any $a,b,c\in V$,
    \[
    (b\succ a)\prec c - b\succ (a\prec c) = (a\prec b)\prec c
    -a\prec (b*c).
    \]
    Similarly, the right-hand side of \eqref{eq:LSym-Chiral-Z1} is symmetric with respect to $a$ and $b$; therefore,
    \[
    (a*b)\succ c -a\succ (b\succ c) =
    (b*a)\succ c - b\succ (a\succ c)
    \]
    for all $a,b,c\in V$. The relations obtained are exactly the identities of an L-dendriform algebra \cite{BaiLiuNi2010}, derived via the dendriform splitting \cite{GubKol-2014} of the left-symmetry identity.
\end{proof}

We also note that any vertex (or Lie chiral) algebra $V$ with a Rota–Baxter operator $R$ in the sense of \cite{Bai_2026} is also a left-symmetric chiral algebra with respect to the operations
\[
a\succ b  = R(a)\cdot R(b),
\quad
(a\oo\sigma b) = [R(a)\oo\sigma b],
\quad a,b\in V.
\]

\section{Variety of Abelian Chiral Algebras}

\begin{definition}\label{defn:AbelianChiral}
Let $\Var$ be a binary operad generated by an $S_2$-space $U=\Var(2)$. A $\Var$-chiral algebra $V$ is called \emph{abelian} if for any $\mu \in U$ its image $X\in \ChEnd_V(2)$ satisfies
\[
X^{z_1,z_2}_{\lambda_1,\lambda_2}(u,v; 1) = 0
\]
for all $u,v\in V$.
\end{definition}

From the condition of $3/2$-linearity \eqref{eq:ChEnd-3/2-lin}, it follows that $X^{z_1,z_2}_{\lambda_1,\lambda_2}(u,v; f) = 0$
in an abelian chiral algebra
for all $f\in \Bbbk [z_{ij}\mid 1\le j<i\le n]= \mathcal D_n1 \subset \mathcal O_n$.

For any $H$-module $V$, we define a linear map $\ChEnd_V(n) \to \End_V(n)$ for $n\ge 1$ by the rule $X\mapsto \widehat X$, where
\[
\widehat X(u_1,\dots, u_n) = X^{z_1,\dots, z_n}_{\lambda_1,\dots , \lambda_n}(u_1,\dots, u_n; \omega _n)\big|_{\lambda_1=\dots=\lambda_{n-1}=0}
\]
for $u_i\in V$. This family of linear maps is not a morphism of operads, but it allows us to reformulate Definition \ref{defn:AbelianChiral} for a chiral algebra in view of relation \eqref{eq:ConfBracket-1}.
Namely, a $\Var$-chiral algebra is abelian if and only if for any $\mu\in \Var(2)$ its image $X \in \ChEnd_V(2)$ satisfies the equality
\[
X^{z_1,z_2}_{\lambda_1,\lambda_2}(u,v; \omega_2) = \widehat X(u,v)
\]
for all $u,v\in V$.

For example, a vertex ($\Var=\Lie$) algebra $V$ defined by the operation
\[
X^{z_1,z_2}_{\lambda_1,\lambda_2}(u,v; \omega_2) = u\cdot v +\int\limits_{0}^{\lambda_1} [u\oo{\sigma }v]\,d\sigma
\]
is abelian if and only if $[u\oo{\sigma }v] = 0$ for all $u,v\in V$. It is well known that such algebras are exactly associative-commutative differential algebras with respect to the product $u\cdot v$ and the derivation~$\partial$.

\begin{remark}\label{rem:S2Abelian}
If $V$ is an abelian $\Var$-chiral algebra, then for any $\mu\in \Var(2)$, its image $X\in \ChEnd_V(2)$ satisfies the condition
\[
\widehat{X^{(12)}} = -{\widehat X}^{(12)}.
\]
\end{remark}

This leads to the following problem statement, which we consider in this paper.
Let $V$ be an abelian $\Var$-chiral algebra, where $\Var$ is a binary quadratic operad generated by an $S_2$-space $U$. The chain of linear maps
\begin{equation}\label{eq:Diag-1}
\begin{CD}
 U\otimes\Bbbk_- @. @. @. \End_V     \\
@V\cong VV @. @. @AA \widehat{\ \cdot\ }A \\
U @>\subset >> \mathcal F_U @>>> \Var @>>> \ChEnd_V
\end{CD}
\end{equation}
defines a linear map from $U\otimes\Bbbk_-$ to $\End_V(2)$. According to Remark \ref{rem:S2Abelian}, this map is $S_2$-invariant, since the vertical arrows in the diagram \eqref{eq:Diag-1} are skew-symmetric.

Recall that the space $V$ is equipped with a linear operator $\partial$. Let us denote by $\widehat U$ the two-component graded space $\widehat U(1)\oplus \widehat U(2)$, where
\[
\widehat U(1)=\Bbbk \partial,
\quad \widehat U(2)=U\otimes \Bbbk_- .
\]
Let us extend the map $U\otimes \Bbbk_-\to \End_V$ given by diagram \eqref{eq:Diag-1} to a symmetric map $\widehat U\to \End_V$ under which $\partial \in \widehat U(1)$ is mapped to $\partial \in \End_V(1)$. This map defines an operad morphism
\[
\widehat V: \mathcal F_{\widehat U}\to \End_V.
\]
This morphism determines the structure of an ordinary algebra on the space $V$ with bilinear operations from the space $U$ and a unary linear operation $\partial$. Let us also denote this algebra by~$\widehat V$.

\begin{remark}
The map $\partial\in \End_V(1)$ is a derivation of the algebra $\widehat V$.
\end{remark}

The problem is to find a necessary and sufficient condition for a given differential algebra $A$ to coincide with the algebra $\widehat V$ for some abelian $\Var$-chiral algebra $V$.

It is clear that to solve this problem it suffices to consider the intersection of the kernels of the morphisms $\widehat V$ for all abelian $\Var$-chiral algebras $V$ as an ideal in the free operad $\mathcal F_{\widehat U}$. The quotient by this ideal defines an operad $\widehat {\Var}$ such that all algebras of the form $\widehat V$ lie in the corresponding variety of $\widehat {\Var}$-algebras, i.e., the diagram
\begin{equation}\label{eq:Diag-2}
\begin{CD}
U\otimes \Bbbk_- @>\subset >> \mathcal F_{\widehat U} @>>> \widehat{\Var} @>>> \End_V     \\
@V\cong VV @. @. @AA \widehat{\ \cdot\ }A \\
U @>\subset >> \mathcal F_U @>>> \Var @>>> \ChEnd_V
\end{CD}
\end{equation}
is commutative for any abelian $\Var$-chiral algebra $V$.

The following statement, which is the main result of this paper, provides the answer to the question regarding the structure of the operad $\widehat {\Var}$.

\begin{theorem}\label{thm:MainThm}
Let $\Var$ be a binary quadratic operad. Then, for any abelian $\Var$-chiral algebra $V$, the algebra $\widehat V$ is a differential $\Var\bullet \Com$-algebra. Conversely, any differential $\Var\bullet\Com$-algebra $A$ can be viewed as an abelian $\Var$-chiral algebra.
\end{theorem}

In other words, the variety $\widehat{\Var}$ coincides with the class of all differential algebras from the variety defined by the operad $\Var\bullet \Com$.

The proof of Theorem~\ref{thm:MainThm} follows from several auxiliary statements. Throughout what follows, $\Var$ is a binary quadratic operad generated by an $S_2$-space $U$. For a given element $f\in U$, we denote $\hat f=f\otimes 1\in \widehat U(2)$.

\begin{proposition}\label{prop:prop2}
Let $V$ be an abelian $\Var$-chiral algebra. Consider arbitrary $f,g\in U$ and their images $X,Y\in \ChEnd_V(2)$. Then
\begin{equation}\label{eq:Comp-1-alpha}
(X\circ_1 Y)_{\lambda_1,\lambda_2,\lambda_3}^{z_1,z_2,z_3}
 (\cdot,\cdot,\cdot;\omega_3) = (\lambda_1+\lambda_2)(\widehat X\circ_1 \widehat Y) + (\widehat X\circ_1 \partial\widehat Y).
\end{equation}
\end{proposition}

\begin{proof}
It suffices to apply formula \eqref{eq:X-1-Yformula} under the assumption that the chiral algebra $V$ is abelian:
\[
(X\circ_1 Y)_{\lambda_1,\lambda_2,\lambda_3}^{z_1,z_2,z_3}
 (a,b,c;\omega_3)
 =((\partial+\lambda_1+\lambda_2)(a\cdot_Y b)\cdot_X c),
\]
as required.
\end{proof}

\begin{proposition}\label{prop:prop3}
Under the conditions of Proposition~\ref{prop:prop2}, the following equalities hold:
\begin{gather}\label{eq:Comp-1-beta}
((X\circ_1 Y)^{(13)})_{\lambda_1,\lambda_2,\lambda_3} ^{z_1,z_2,z_3}
 (\cdot,\cdot,\cdot;\omega_3) =
 \lambda_1(\widehat X\circ_1 \widehat Y)^{(13)} +
 ((\widehat X\circ_1 \widehat Y)\circ_1\partial)^{(13)},
 \\
 \label{eq:Comp-1-gamma}
((X\circ_1 Y)^{(23)})_{\lambda_1,\lambda_2,\lambda_3} ^{z_1,z_2,z_3}
 (\cdot,\cdot,\cdot;\omega_3) =
 \lambda_2(\widehat X\circ_1 \widehat Y)^{(23)} +
 ((\widehat X\circ_1 \widehat Y)\circ_2\partial)^{(23)}.
\end{gather}
\end{proposition}

\begin{proof}
According to \eqref{eq:ChEnd-SymAct},
\begin{multline*}
((X\circ_1 Y)^{(13)})_{\lambda_1,\lambda_2,\lambda_3} ^{z_1,z_2,z_3}
(a,b,c; \omega_3)
=
(X\circ_1 Y)_{\lambda_3,\lambda_2,\lambda_1}^{z_1,z_2,z_3}
(c,b,a; \omega_3^{(13)}) \\
=-((\partial +\lambda_2+\lambda_3)(c\cdot_Y b)\cdot_X a)
=-\lambda_2((c\cdot_Y b)\cdot_X a)
-\lambda_3((c\cdot_Y b)\cdot_X a)
-(\partial (c\cdot_Y b)\cdot_X a).
\end{multline*}
Recall that multiplication by $\lambda_3$ is equivalent to the action of $-\partial-\lambda_1-\lambda_2$ by the definition of the operad of chiral endomorphisms, and since $\partial$ is a derivation, we have
\[
((X\circ_1 Y)^{(13)})_{\lambda_1,\lambda_2,\lambda_3} ^{z_1,z_2,z_3}
(a,b,c; \omega_3)
=
\lambda_1((c\cdot_Y b)\cdot_X a) +
((c\cdot_Y b)\cdot_X \partial a)
=
(\widehat X\circ_1 \widehat Y)^{(13)}(\partial a,b,c),
\]
as required to prove \eqref{eq:Comp-1-beta}.
Relation \eqref{eq:Comp-1-gamma} is proved similarly.
\end{proof}

Let us return to the proof of Theorem~\ref{thm:MainThm}.
As above, let $(\mu_j)_{j\in J}$ be a basis of the space $U$.
An arbitrary element $p\in \mathcal F_U(3)$ can be uniquely represented in the form
\begin{equation}\label{eq:F(3)generic}
p = \sum\limits_{i,j\in J}
\alpha_{ij}(\mu_i\circ_1 \mu_j) + \beta_{ij}(\mu_i\circ_1 \mu_j)^{(13)} +
\gamma_{ij}(\mu_i\circ_1 \mu_j)^{(23)},
\end{equation}
where $\alpha_{ij},\beta_{ij},\gamma_{ij}\in \Bbbk $.

Let us compute the image of $p$ under the morphism $\Phi$ given by formula \eqref{eq:Phi-morphism} for $\mathcal P=\Com$, using the notation from Example~\ref{exmp:Pois}.
By definition,
$\Phi(\mu_i) = e^*\otimes \hat\mu_i$.
Consequently,
\begin{equation}\label{eq:3-monom-1}
\Phi(\mu_i\circ_1\mu_j) = (e^*\circ _1 e^*)\otimes (\hat\mu_i \circ _1 \hat\mu_j)
=(e_1-e_2)\otimes (\hat\mu_i \circ _1 \hat\mu_j).
\end{equation}
Hence we obtain
\begin{gather}
    \Phi((\mu_i\circ_1\mu_j)^{(13)}) = e_1\otimes (\hat\mu_i \circ _1 \hat\mu_j)^{(13)},   \label{eq:3-monom-2}\\
\Phi((\mu_i\circ_1\mu_j)^{(23)}) = -e_2\otimes (\hat\mu_i \circ _1 \hat\mu_j)^{(23)}.   \label{eq:3-monom-3}
\end{gather}
Thus, it follows from \eqref{eq:3-monom-1}--\eqref{eq:3-monom-3} that
\[
\Phi(p) = e_1\otimes p_1 - e_2\otimes p_2,
\]
where
\begin{gather}
p_1 = \sum\limits_{i,j\in J}
\alpha_{ij}(\hat\mu_i\circ_1 \hat\mu_j)
+ \beta_{ij}(\hat\mu_i\circ_1 \hat\mu_j)^{(13)}, \label{eq:p-1}\\
p_2 = \sum\limits_{i,j\in J}
\alpha_{ij}(\hat\mu_i\circ_1 \hat\mu_j)+\gamma_{ij}(\hat\mu_i\circ_1 \hat\mu_j)^{(23)}.\label{eq:p-2}
\end{gather}
The element $p$ is an identity on the class of all $\Var$-algebras if and only if $p_1$ and $p_2$ are identities on the class of all $\Var\bullet\Com$-algebras.

Now let $V$ be an abelian $\Var$-chiral algebra.
Denote by $X_j$ the image of the generator $\mu_j$ in $\ChEnd_V(2)$,
then $\widehat X_j$ is the image of $\hat\mu_j$ in $\End_V$, see diagram \eqref{eq:Diag-2}.

If $p\in \mathcal F_U(3)$ of the form \eqref{eq:F(3)generic} is an identity of the variety $\Var$,
then
\[
\sum\limits_{i,j\in J}
\alpha_{ij}(X_i\circ_1 X_j) + \beta_{ij}(X_i\circ_1 X_j)^{(13)} +
\gamma_{ij}(X_i\circ_1 X_j)^{(23)} =0
\]
in $\ChEnd_V(3)$.
On the other hand, from Propositions \ref{prop:prop2} and \ref{prop:prop3}
it follows that
\begin{multline}\label{eq:Splitting-P}
p^{z_1,z_2,z_3}_{\lambda_1,\lambda_2,\lambda_3}
(\cdot,\cdot,\cdot; \omega_3)
=\lambda_1 \bigg(
\sum\limits_{i,j\in J}
 \alpha_{ij}(\widehat X_i\circ_1 \widehat X_j)
+ \beta_{ij} (\widehat X_i\circ_1 \widehat X_j)^{(13)}
\bigg)
\\
+\lambda_2 \bigg (
\sum\limits_{i,j\in J}
\alpha_{ij}(\widehat X_i\circ_1 \widehat X_j)
+
\gamma_{ij} (\widehat X_i\circ_1 \widehat X_j)^{(23)}
\bigg)
\\
+\sum\limits_{i,j\in J}
\alpha_{ij}(\widehat X_i\otimes \partial\widehat X_j)
+ \beta_{ij}((\widehat X_i\circ_1 \widehat X_j)\circ_1\partial )^{(13)}
+ \gamma_{ij}((\widehat X_i\circ_1 \widehat X_j)\circ_2\partial )^{(23)}.
\end{multline}
The coefficients of $\lambda_1$ and $\lambda_2$ coincide with the images of $p_1$ and $p_2$ in $\End_V$, respectively. Consequently, the algebra $\widehat V$ satisfies all the defining identities of the variety of $\Var\bullet\Com$-algebras.

Conversely, let $A$ be a differential $\Var\bullet\Com$-algebra.
The operations in this algebra are defined by bilinear maps $\widehat X_j \in \End_A(2)$.
We define an operad morphism $\mathcal F_U \to \ChEnd_A$ by the rule
\begin{equation}\label{eq:Diff-to-Chiral}
\mu_j\mapsto X_j,\quad (X_j)^{z_1,z_2}_{\lambda_1,\lambda_2}(a,b; \omega_2) = 1\otimes_H \widehat X_j(a,b)
\end{equation}
for all $a,b\in A$.
Consequently, for any identity $p$ of the variety $\Var$, the terms in formula \eqref{eq:Splitting-P} containing $\lambda_1$ and $\lambda_2$ are equal to zero as the images of the functions $p_1,p_2\in \mathcal F_{\widehat U}(3)$.
The remaining terms form the sum
\[
q =
\sum\limits_{i,j\in J}
\alpha_{ij}(\widehat X_i\otimes \partial\widehat X_j)
+ \beta_{ij}((\widehat X_i\circ_1 \widehat X_j)\circ_1\partial )^{(13)}
+ \gamma_{ij}((\widehat X_i\circ_1 \widehat X_j)\circ_2\partial )^{(23)}.
\]
Let us compute $q(a,b,c)$ for $a,b,c\in A$, using the fact that $\partial$ is a derivation on the algebra $A$:
\begin{multline*}
q(a,b,c) =
\sum\limits_{i,j\in J}
\alpha _{ij} \widehat X_i(\partial \widehat X_j(a,b),c)
+
\beta _{ij} \widehat X_i(\widehat X_j(c,b),\partial a)
+
\gamma_{ij} \widehat X_i(\widehat X_j(a,c),\partial b) \\
=
\sum\limits_{i,j\in J}
\alpha _{ij} \widehat X_i(\widehat X_j(\partial a,b),c)
+
\beta _{ij} \widehat X_i(\widehat X_j(c,b),\partial a) \\
+
\sum\limits_{i,j\in J}
\alpha _{ij} \widehat X_i(\partial \widehat X_j(a,\partial b),c)
+
\gamma_{ij} \widehat X_i(\widehat X_j(a,c),\partial b) \\
=
p_1(\partial a,b,c) + p_2(a,\partial b, c) = 0\in A.
\end{multline*}
Consequently, all the defining identities of the variety $\Var$ lie in the kernel of the constructed operad morphism \eqref{eq:Diff-to-Chiral}, i.e., $A$ is a $\Var$-chiral algebra. \qed


\noindent \textsc{Sobolev Institute of Mathematics, Novosibirsk (Russia)}\\
\textit{Email address:} \texttt{i.v.dudin@math.nsc.ru}

\vspace{0.5cm}

\noindent \textsc{Sobolev Institute of Mathematics, Novosibirsk (Russia)}\\
\textit{Email address:} \texttt{pavelsk77@gmail.com}

\end{document}